\documentclass[11 pt]{amsart}
\usepackage{amscd,amsfonts,amssymb,amsmath}
\usepackage{hyperref}
\usepackage{epsfig}
\usepackage{mathtools}
\usepackage{tabu}
\usepackage{tikz-cd}
\newtheorem{theorem}{Theorem}[section]
\newtheorem{corollary}[theorem]{Corollary}
\newtheorem{lemma}[theorem]{Lemma}
\newtheorem{proposition}[theorem]{Proposition}
\theoremstyle{definition}

\newtheorem{question}[theorem]{Question}

\numberwithin{equation}{subsection}

\usepackage[all,cmtip]{xy}

\usepackage{graphicx} 

\newcommand{\Maps}{\operatorname{Map}}
\newcommand{\Aut}{\operatorname{Aut}}
\newcommand{\Autg}{\operatorname{\textbf{Aut}}}
\newcommand{\Ext}{\operatorname{Ext}}
\newcommand{\Bilin}{\operatorname{Bilin}}

\newcommand{\Ker}{\operatorname{Ker}}
\newcommand{\im}{\operatorname{Im}}

\newcommand\Norm{{\scriptstyle\mathrm{N}}}

\newcommand{\Inn}{\operatorname{Inn}}

\newcommand{\Ho}{\operatorname{H}}
\newcommand{\Hog}{\operatorname{\textbf{H}}}

\newcommand{\B}{\operatorname{B}}

\newcommand{\Z}{\operatorname{Z}}
\newcommand{\Zg}{\operatorname{\textbf{Z}}}
\newcommand{\Bg}{\operatorname{\textbf{B}}}

\newcommand{\id}{\mathrm{id}}

\begin{document}

\title[Wells type exact sequence for solutions of the Yang--Baxter equation]{A Wells type exact sequence for non-degenerate unitary solutions of the Yang--Baxter equation}

\author{Valeriy Bardakov}
\address{Sobolev Institute of Mathematics and Novosibirsk State University, Novosibirsk 630090, Russia.\linebreak
Novosibirsk State Agrarian University, Dobrolyubova street, 160, Novosibirsk, 630039, Russia.\linebreak
Regional Scientific and Educational Mathematical Center, Tomsk State University, pr. Lenina, 36, Tomsk, 634050, Russia}
\email{bardakov@math.nsc.ru}
\author{Mahender Singh}
\address{Department of Mathematical Sciences, Indian Institute of Science Education and Research (IISER) Mohali, Sector 81,  S. A. S. Nagar, P. O. Manauli, Punjab 140306, India.}
\email{mahender@iisermohali.ac.in}

\subjclass[2010]{Primary 16T25, 20N02; Secondary 55N35, 57M27, 20J05}
\keywords{Brace, cycle set cohomology, linear cycle set, extension, group cohomology, Yang--Baxter equation}

\begin{abstract}
Cycle sets are known to give non-degenerate unitary solutions of the Yang--Baxter equation and linear cycle sets are enriched versions of these algebraic systems. The paper explores the recently developed cohomology and extension theory for linear cycle sets. We derive a four term exact sequence relating 1-cocycles, second cohomology and certain groups of automorphisms arising from central extensions of linear cycle sets. This is an analogue of a similar exact sequence for group extensions known due to Wells. We also relate the exact sequence for linear cycle sets with that for their underlying abelian groups via the forgetful functor and also discuss generalities on dynamical 2-cocycles.
\end{abstract}
\maketitle

\section{Introduction}\label{introduction}
The quantum Yang--Baxter equation is a fundamental equation arising in theoretical physics and has deep connections with mathematics specially braid groups and knot theory. A solution of the quantum Yang--Baxter equation is a linear map $R:V \otimes V \to V\otimes V$ satisfying $$R_{12}R_{13}R_{23} = R_{23}R_{13}R_{12},$$ where 
$V$ is a vector space over a field and $R_{ij}:V \otimes V \otimes V  \to  V \otimes V \otimes V$ acts as $R$ on the $(i,j)$ tensor factor and as the identity on the remaining factor. If $F: V \otimes V \to V \otimes V$ is the flip operator $F(v \otimes w) = w \otimes v$, then  $R:V \otimes V \to V\otimes V$ is a solution of the quantum Yang--Baxter equation if and only if $\overline{R} = F \circ R$ satisfies the braid relation
$$\overline{R}_{12}\overline{R}_{23}\overline{R}_{12} = \overline{R}_{23}\overline{R}_{12}\overline{R}_{23},$$
in which case one says that $\overline{R}$ is a solution of the  Yang--Baxter equation.  Topologically, the braid relation is simply the third Reidemeister move of planar diagrams of links as shown in Figure \ref{fig1}.
 \begin{figure}[!ht]
 \begin{center}
\includegraphics[height=2cm, width=4cm]{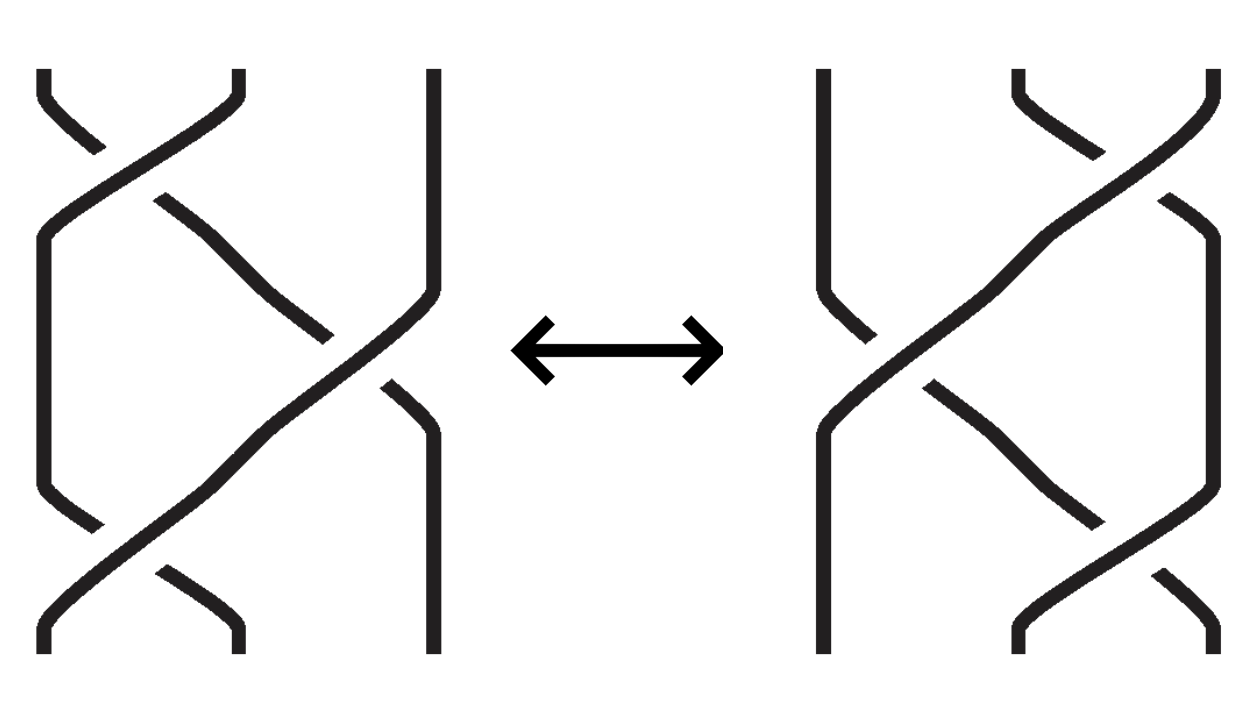}
\end{center}
\caption{The braid relation} \label{fig1}
\end{figure}

 If $X$ is a basis of the vector space $V$, then a map $r: X \times X \to  X \times X$ satisfying
$r_{12}r_{23}r_{12}= r_{23}r_{12}r_{23}$ induces a solution of the Yang--Baxter equation. In this case, we say that $(X,r)$ is a set-theoretic solution of the Yang--Baxter equation.  Writing $r(x, y) = (\sigma_x(y), \tau_y(x))$ for $x, y \in X$, we say that the solution $r$ is non-degenerate if $\sigma_x$ and $\tau_x$ are invertible for all $x \in X$. The problem of finding these set-theoretic solutions was posed by Drinfeld \cite{Drinfeld} and has attracted a lot of attention.
\par

A \emph{(left)  cycle set}, as defined by Rump \cite{Rump2005}, is a non-empty set $X$ with a binary operation $\cdot$ having bijective left translations $X \to X$, $x \mapsto y \cdot x$, 
and satisfying the relation
	\begin{align}\label{E:Cyclic}
	(x\cdot y)\cdot (x\cdot z)=(y\cdot x)\cdot (y\cdot z)
\end{align}
for all $x, y, z \in X$. A cycle set is \emph{non-degenerate} if the squaring map $a \mapsto a \cdot a$ is invertible. It is known that every finite cycle set is non-degenerate \cite[Theorem 2]{Rump2005}. Rump showed that cycle sets are in bijection with non-degenerate unitary set-theoretic solutions of the Yang--Baxter equation. These solutions give a rich class of structures and are connected with semigroups of special type, Bieberbach groups \cite{IvanovaBergh}, biquandles \cite{FJSK, KM2005}, colourings of plane curves \cite{ESS1999}, Hopf algebras \cite{EtingofGelaki1998} and Garside groups \cite{Chouraqui}, to name a few. Special solutions, particularly, the ones possessing self-distributivity are intimately connected to invariants of knots and links in the 3-space and thickened surfaces \cite{FJSK, KM2005, LebedJKTR2018}. Cycle sets have been proved to be very useful in understanding the structure of  solutions of the Yang--Baxter equation and for obtaining general classification results. Cycle sets as braces give only unitary solutions whereas skew braces, racks, bi-quandles etc. give general solutions. The structure of cycle sets is still far from being completely understood, and many important questions on the topic are yet not answered. The reader is referred to \cite{CJO2010, CJO2014, Chouraqui, Dehornoy2015, ESS1999, GI2018, LV2016, LV2017, YanZhu2000, RumpJA2007, Smok2018, Smoktunowicz2018, Soloviev2000} for some recent works. 
\par

A \emph{(left) brace} is an abelian group $(A,+)$ with an additional group operation $\circ$ such that
	\begin{align}\label{E:Brace}
	a\circ (b+c)+a=a\circ b+a\circ c
\end{align}
holds for all $a,b,c \in A$. Braces were introduced by Rump \cite{RumpJA2007} in a slightly different but equivalent form where he showed that these algebraic systems give set-theoretic solutions of the Yang--Baxter equation. The preceding definition is due to Ced{\'o}, Jespers and Okni{\'n}ski \cite{CJO2014}. Each abelian group is trivially a brace with $a+b=a \circ b$. In addition, regular rings give a large supply of braces. Relations between the additive and the multiplicative groups of a brace have been explored in many recent works, for example, \cite{CedoSmokVendramin, GorshkovNasybullov, Nasybullov2019}.
\par

A \emph{(left) linear cycle set} is a cycle set $(X,\cdot)$ with an abelian group operation $+$ satisfying the conditions
\begin{equation}\label{E:LinCyclic}
x\cdot (y+z) =x\cdot y+x\cdot z
\end{equation}
and
\begin{equation}\label{E:LinCyclic2}
(x+y)\cdot z =(x\cdot y)\cdot (x\cdot z)
\end{equation}
for all $x, y, z\in X$. This notion goes back to Rump \cite{RumpJA2007}, who showed it to be equivalent to the brace structure via the relation $$x\cdot y=x^{-1}\circ(x+y),$$ 
where $x^{-1}$ is inverse with respect to $\circ$. An abelian group can be viewed as a linear cycle set by taking $x\cdot y=y$ for all $x, y \in X$, and referred as a {\it trivial} linear cycle set. Rump \cite{RumpJA2007} showed that linear cycle sets are closely related to radical rings.
\par

It was pointed out in \cite{BCJO2017} that an extension theory for cycle sets (equivalently braces) would be crucial for a classification of these objects. This led to development of an extension theory by Lebed and Vendramin \cite{LV2016, LV2017}. A homology and cohomology theory for linear cycle sets (and hence for braces) was developed recently in \cite{LV2016}.  As in case of groups, Lie algebras, quandles or any nice algebraic system, the second cohomology groups were shown to classify central cycle set extensions.  A cohomology theory for general cycle sets was developed in \cite{LV2017}. We will follow the linear cycle set language of \cite{LV2016} since it gives a neat construction of cohomology and extension theory. It is worth noting that the right and the two-sided analogues of braces and cycle sets can be defined analogously and have been considered in the literature.
\par

In this paper, we derive an exact sequence relating 1-cocycles, certain group of automorphisms and second cohomology groups of linear cycle sets. This can be thought of as a linear cycle set analogue of a fundamental exact sequence for groups due to Wells \cite{Wells1971}.  For notational convenience, sometimes, we will denote the value of a map $\phi$ at a point $x$ by $\phi_x$. We use the notation $\Aut$,  $\Z^1$ and $\Ho^2$ to denote group of automorphisms, group of  1-cocycles and second cohomology group of a linear cycle set, respectively. To distinguish linear cycle sets from groups, we use the bold notation $\Autg$, $\Zg^1$ and $\Hog^2$ to denote group of automorphisms, group of  1-cocycles and second cohomology of a group, respectively.
\par

Section \ref{prelim} contains preliminaries and some basic results.  We prove that there is a natural group homomorphism from the second linear cycle set cohomology to the second symmetric cohomology of the underlying abelian group (Proposition \ref{homo-cycle-group-cohomology}) and also examine this homomorphism for trivial cycle sets (Proposition  \ref{trivial-cycle-group-cohomology}). Section \ref{sec-action-auto-cohomo} prepares the foundation for the main result. Given a linear cycle set $(X, \cdot, +)$ and an abelian group $A$, we define an action of $\Aut(X) \times \Autg(A)$ on $\Ho_\Norm^2(X; A)$. As a consequence, we obtain a lower bound on the size of $\Ho_\Norm^2(X; A)$ (Corollary \ref{cor-order-cohom}). In Section \ref{sec-exact-seq}, we prove our main theorem (Theorem \ref{abelian-main-theorem}) that associates to each central extension of linear cycle sets a four term exact sequence relating group of 1-cocycles, certain group of automorphisms and second cohomology groups. Section \ref{properties of map theta} explores properties of the important connecting map in this exact sequence (Theorem \ref{main-thm-3}). In Section \ref{comparision-sequences}, we relate the exact sequence with the corresponding Wells exact sequence for the underlying extension of abelian groups via the forgetful functor (Theorem \ref{comparison-wells}). Finally, in Section \ref{dynamical-2-cocycle}, we discuss some generalities on bi-groupoids and dynamical extensions of (linear) cycle sets.
\medskip

\section{Preliminaries and some basic results}\label{prelim}
Recall that a bi-groupoid is a non-empty set with two binary algebraic operations. We begin with the following immediate observation.

\begin{lemma}
Let $(X, \cdot, *)$ be a bi-groupoid and   $S : X \times X \to X \times X$ given by $S(x, y) = (x \cdot y, y * x)$ for $x, y \in X$. Then the following hold:
\begin{enumerate}
\item The pair $(X, S)$ is a set theoretic solution of the Yang--Baxter equation if and only if the equalities
\begin{eqnarray*}
(x \cdot (y \cdot z))  &=& (x \cdot y) \cdot ((y*x)\cdot z),\\
((y\cdot z)* x) \cdot (z *y)  &=&  ((y*x) \cdot z) *(x \cdot y)~\textrm{and}\\
(z*y) *((y \cdot z) *x)  &=&  z* (y * x)
\end{eqnarray*}
hold for all $x, y, z \in X$. 
\item If $x \cdot y = y$ for all $x, y \in X$, then the  pair $(X, S)$  is a set theoretic solution of the Yang--Baxter equation if and only if the operation $*$ is left distributive, i.e.
\begin{equation}\label{rack-eqn}
z *(y * x) = (z*y) * (z* x)
\end{equation}
 for all $x, y, z \in X$.
\end{enumerate}
\end{lemma}

Recall that $(X, *)$ is a {\it (left) rack} if the maps $x \mapsto y*x$ are bijections and \eqref{rack-eqn} holds for all $x,y,z \in X$. Thus, assertion (2) of the preceding lemma gives a non-degenerate solution of the Yang--Baxter equation if and only if $(X, *)$ is a rack. Racks are useful in defining invariants of framed links in the 3-space. We look for conditions under which a rack is a cycle set. Following \cite{LebedMortier}, we say that a rack $X$ is {\it abelian} if 
$$
x * (y * z) = y * (x * z)
$$
for all $x, y, z \in X$. Note that this condition is equivalent to the group $$\Inn(X)= \langle S_x, ~x \in X~|~ S_x(y):= x*y,~ x, y \in X\rangle $$  of inner automorphisms of $X$ being abelian.

\begin{proposition}
If $(X, *)$ is a rack such that $\Inn(X)$ is abelian, then it is a cycle set.
\end{proposition}

\begin{proof}
It follows from the rack axiom that
$$
x * (y * z) = (x * y) * (x  * z)$$
and
$$y * (x * z) = (y * x) * (y  * z)
$$
for all $x, y, z \in X$. Since $\Inn(X)$ is abelian, we get $$(x * y) * (x  * z)=  (y * x) * (y  * z),$$ which is desired.
\end{proof}
\bigskip

\subsection{Cohomology and extensions of linear cycle sets}
 A \emph{morphism} between linear cycle sets $X$ and $Y$ is a map $\varphi \colon X\to Y$ satisfying  $\varphi(x+x')=\varphi(x)+\varphi(x')$ and $\varphi(x \cdot x')=\varphi(x) \cdot \varphi(x')$ for all $x,x'\in X$. The \emph{kernel} of $\varphi$ is defined by $\Ker (\varphi) = \varphi^{-1}(0)$. The notion of \emph{image}, of a \emph{short exact sequence} of linear cycle sets, and of \emph{linear cycle subsets} are defined in the usual manner. 
\par

Two  linear cycle set extensions $A \overset{i}{\rightarrowtail} E \overset{\pi}{\twoheadrightarrow} X$ and $A \overset{i'}{\rightarrowtail} E' \overset{\pi'}{\twoheadrightarrow} X$ are called \emph{equivalent} if there exists a linear cycle set isomorphism $\varphi \colon E \to E'$ such that the diagram 

$$
\xymatrix{
A  \ar[r]^{i}  \ar[d]^{\id} & E \ar[d]^{\varphi} \ar[r]^{\pi} & X  \ar[d]^{\id} \\
A  \ar[r]^{i'} & E' \ar[r]^{\pi'} & X}
$$
commutes. A cohomology theory for linear cycle sets is developed in the recent works of Lebed and Vendramin  \cite{LV2016, LV2017}. Following \cite{LV2016}, a {\it 2-cocycle} for a linear cycle set  $(X, \cdot, +)$ with coefficients in the (additively written) abelian group $A$ consists of two maps $f,g \colon X \times X \to A$ satisfying the  conditions
\begin{eqnarray}
g(x,y) &=& g(y,x),\label{g-symmetric}\\
g(x,y) + g(x+y,z) &=&  g(y,z) + g(x,y+z),\label{CocycleFull3}\\
f(x+y,z) &=&  f(x \cdot y,x \cdot z) + f(x,z)~\textrm{and}\label{CocycleFull}\\
f(x,y+z) - f(x,y) - f(x,z) &=&  g(x \cdot y, x \cdot z) - g(y,z)\label{CocycleFull2}
\end{eqnarray}
for all $x, y, z \in X$. Note that, if $(f,g)$ is a $2$-cocycle of a linear cycle set $(X, \cdot, +)$ with coefficients in an abelian group $A$, then conditions \eqref{g-symmetric}-\eqref{CocycleFull3}-\eqref{CocycleFull}\eqref{CocycleFull2} imply that
\begin{align*}
&f(0,x) = f(x,0) = 0~\textrm{and}\\
&g(0,x) = g(x,0) = g(0,0)
\end{align*}
for all $x \in X$. A pair of maps $f,g \colon X \times X \to A$ is called a {\it 2-coboundary}  if there exists a map $\lambda \colon X \to A$ such that
\begin{eqnarray}
	f(x,y) &=& \lambda (x \cdot y) - \lambda(y)~\textrm{and}\label{Coboundary1}\\
	g(x,y) &=& \lambda (x + y) - \lambda(x) - \lambda(y)\label{Coboundary2}
\end{eqnarray}
for all $x, y \in X$. 
\par
A $2$-cocycle $(f,g)$ is called normalised if $g(0,0) = 0$, whereas a $2$-coboundary $(f,g)$ is called normalised if the map
$\lambda: X \to A$ satisfy $\lambda(0) = 0$. We denote the group of normalised 2-cocycles by $\Z_\Norm^2(X; A)$, and the group of normalised 2-coboundaries by $\B_\Norm^2(X; A)$. The quotient $\Z_\Norm^2(X; A)/\B_\Norm^2(X; A)$ is the {\it normalised cohomology} group  $\Ho_\Norm^2(X; A)$ of $X$ with coefficients in $A$. We shall also need the group of normalised 1-cocycles defined as
\begin{equation}\label{1-cocycles}
\Z_\Norm^1(X; A)= \big\{\lambda: X \to A~|~ \lambda (x + y) = \lambda(x) + \lambda(y)~\textrm{and}~\lambda (x \cdot y) = \lambda(y)~\textrm{for all}~x, y \in X \big\}.
\end{equation}

The following is an analogue of a similar classical result for groups \cite[Lemma 5.2]{LV2016}.

\begin{lemma}\label{linear-cycle-structure}
Let $(X, \cdot, +)$ be a linear cycle set, $A$ an abelian group and $f,g \colon X\times X\to A$ two maps. Then the set $X \times A$ with the operations
\begin{align*}
(x, a) + (y, b) &= \big( x + y, a + b + g(x,y)\big) ~\textrm{and}\\
(x, a) \cdot (y, b) &= \big(x \cdot y, b + f(x,y)\big)
\end{align*}
for $a,b \in A$,\, $x,y \in X$, is a linear cycle set if and only if $(f,g)$ is a $2$-cocycle.
\end{lemma}

The linear cycle set of Lemma \ref{linear-cycle-structure} is denoted by $X \oplus_{f,g} A$. A reformulation of Lemma \ref{linear-cycle-structure} for braces is as follows \cite[Lemma 5.4]{LV2016}.

\begin{lemma}\label{brace-structure-cocycle}
	Let $(X, \circ, +)$ be a brace, $A$ be an abelian group, and $f,g\colon X\times X\to A$ be two maps. Then the set $X \times A$ with the operations
\begin{align*}
	(x, a)+(y, b)&=\big(x+y, a+b+g(x,y)\big)~\textrm{and}\\
	(x, a) \circ (y, b)&=\big(x \circ y, a+b+f(x,y)\big)
	\end{align*}
for $a,b\in A$,\, $x, y\in X$, is a brace if and only if for the corresponding linear cycle set $(X, \cdot, +)$, the pair $(\overline{f}, g)$
is a $2$-cocycle, where
\begin{align}
\overline{f}(x,y) = -f(x, x \cdot y) + g(x,y)
\end{align}
for all $x, y \in X$.
\end{lemma}

A linear cycle subset $X'$ of a linear cycle set $X$ is called \emph{central} if $x \cdot x' = x'$ and $x' \cdot x = x$ for all $x \in X$ and $x' \in X'$. A \emph{central extension} of a linear cycle set $(X, \cdot, +)$ by an abelian group $A$ is the datum of a short exact sequence of linear cycle sets
\begin{align}\label{ShortExactSeq}
&0 \to A \overset{i}{\to} E \overset{\pi}{\to} X \to 0,
\end{align}
where $A$ is endowed with the trivial cycle set structure, and its image $i (A)$ is central in $E$.  Notice that an extension 
\begin{equation}\label{group-ext}
0 \to A \to E \to X \to 0
\end{equation}
of abelian groups can be thought of as a central extension of linear cycle sets by viewing each group as a trivial linear cycle set.  Let  $\Ext(X,A)$ denote the set of equivalence classes of central extensions of $X$ by $A$. The linear cycle set  $X \oplus_{f,g} A$ from Lemma~\ref{linear-cycle-structure} is a central extension of $X$ by $A$ in the obvious way. More precisely, we have a central extension of linear cycle sets
$$0 \to A \overset{i}{\to} X \oplus_{f,g} A \overset{\pi}{\to} X \to 0,$$
where $i(a)=(0, a)$ and $\pi(x, a)=x$ for all $a \in A$ and $x \in X$. Trivially, the underlying extension 
$$0 \to A \overset{i}{\to} (X \oplus_{f,g} A, +) \overset{\pi}{\to} (X, +) \to 0$$
of abelian groups is also central. As in case of groups, all central linear cycle set extensions of $X$ by $A$ arise in this manner \cite[Lemma 5.6]{LV2016}.

\begin{lemma}\label{extension-to-cocycle}
Let $0 \to A \overset{i}{\to} E \overset{\pi}{\to} X \to 0$ be a central linear cycle set  extension and $s \colon X \to E$ be a set-theoretic section of~$\pi$. 
\begin{enumerate}
\item The maps $f, g\colon X \times X \to E$ defined by
	\begin{eqnarray*}
f (x,y) &=& s(x) \cdot s(y) - s(x \cdot y) ~\textrm{and}\\
g  (x,y) &=& s(x) + s(y) - s(x + y)
\end{eqnarray*}
take values in $i(A)$ and $(f,g)$ is a 2-cocycle.
\item The cocycle above is normalised if and only if  $s(0)=0$.
\item Extensions $E$ and $X \oplus_{f,g} A$ are equivalent. 
\item A cocycle $(f',g')$ obtained from another section~$s'$ of~$\pi$ is cohomologous to $(f,g)$. If both cocycles are normalised, then they are cohomologous in the normalised sense.
\end{enumerate}  
\end{lemma}

Lemma \ref{extension-to-cocycle} yields a bijective correspondence
\begin{equation}\label{ext-equiv-cohomo}
\Ext(X,A) \longleftrightarrow \Ho_{\Norm}^2(X; A).
\end{equation}
Thus, central extensions of  linear cycle sets (and hence of braces) are completely determined by their second normalised cohomology groups. 

The rest of the paper centres around the extension  $X \oplus_{f,g} A$ of a linear cycle set $(X, \cdot, +)$ by an abelian group $A$ with respect to a given $2$-cocycle  $(f,g)$. Thus, in view of lemmas \ref{linear-cycle-structure} and \ref{extension-to-cocycle}, we would only use the 2-cocycle conditions \eqref{g-symmetric}--\eqref{CocycleFull2} rather than the defining axioms of a linear cycle set.
\medskip

\subsection{Homomorphism from linear cycle set cohomology to group cohomology}

Let $G$ be a group, $A$ a $G$-module and $\Hog^2(G; A)$ the second group cohomology of $G$ with coefficients in $A$. It is well-known that  $\Hog^2(G; A)$ classifies group extensions of $G$ by $A$ inducing the given action of $G$ on $A$ \cite{Brown1981}.  Recall that a group theoretical 2-cocycle $g: G \times G \to A$ satisfying \eqref{g-symmetric} is called a {\it symmetric 2-cocycle}. Let $\Hog^2_{sym}(G; A)$ be the subgroup of $\Hog^2(G; A)$ consisting of cohomology classes of  group theoretical symmetric 2-cocycles (that is, maps satisfying conditions \eqref{g-symmetric} and \eqref{CocycleFull3}). An easy check shows that if both $G$ and $A$ are abelian groups, then  $\Hog^2_{sym}(G; A)$ classifies extensions of the form $$0 \to A \to E \to X \to 0,$$ where $E$ is an abelian group. It follows that $A$ is necessarily a trivial $G$-module in this case. There is a natural group homomorphism from linear cycle set cohomology to symmetric group cohomology.

\begin{proposition}\label{homo-cycle-group-cohomology}
Let $(X, \cdot, +)$ be a linear cycle set and $A$ an abelian group viewed as a trivial $(X, +)$-module. Then there is a group homomorphism
$$\Lambda: \Ho^2_N(X; A) \to \Hog^2_{sym}((X, +); A)$$ given by $\Lambda[(f, g)]= [g]$.
\end{proposition}

\begin{proof}
Given a  normalised $2$-cocycle $(f,g) \in \Z_\Norm^2(X; A)$ for the linear cycle set  $(X, \cdot, +)$ with coefficients in the abelian group $A$, conditions \eqref{g-symmetric} and
\eqref{CocycleFull3} imply that $g$ is a group theoretical symmetric $2$-cocycle of the abelian group $(X,+)$ with coefficients in the abelian group $A$. Further, if $(f,g) \in \B_\Norm^2(X; A)$, then condition \eqref{Coboundary2} imply that $g$ is a group theoretical 2-coboundary. Thus, there is a well-defined map $\Lambda: \Ho^2_N(X; A) \to \Hog^2_{sym}((X, +); A)$ given by $\Lambda[(f, g)]= [g]$. That $\Lambda$ is a group homomorphism follows from the fact that addition of 2-cocycles is point-wise for both linear cycle sets and groups.
\end{proof}

Given two abelian groups $G$ and $A$, let $\Bilin(G \times G, A)$ denote the group of bilinear maps from $G \times G$ to $A$. For trivial cycle sets, the map $\Lambda$ is surjective and its kernel can be determined precisely.

\begin{proposition}\label{trivial-cycle-group-cohomology}
Let $(X, \cdot, +)$ be a trivial linear cycle set and $A$ an abelian group viewed as a trivial $(X, +)$-module. Then
$$\Ho_{\Norm}^2(X; A) \cong \Bilin((X,+) \times (X,+), A) \times \Hog^2_{sym}((X, +); A).$$
\end{proposition}

\begin{proof}
We begin by noting that if  $f,g \colon X \times X \to A$ is a 2-cocycle of the trivial linear cycle set $(X, \cdot, +)$, then conditions \eqref{g-symmetric}-\eqref{CocycleFull3} imply that $g$ is a group theoretical symmetric 2-cocycle, whereas conditions  \eqref{CocycleFull}-\eqref{CocycleFull2} imply that $f$ is a bilinear map, that is, $f \in \Bilin((X,+) \times (X,+), A)$. Further, by conditions \eqref{Coboundary1}-\eqref{Coboundary2}, $(f, g)$ is a 2-coboundary if there exists a map $\lambda \colon X \to A$ such that $$	f(x,y) =0$$ and $$g(x,y) =\lambda (x + y) - \lambda(x) - \lambda(y)$$
for all $x, y \in X$. Thus, it follows that $$\Ho_{\Norm}^2(X; A) \cong \Bilin((X,+) \times (X,+), A) \times \Hog^2_{sym}((X, +); A),$$
and the map $\Lambda$ is simply projection onto the second factor.
\end{proof}

Given two abelian groups $(X, +)$ and $A$ where $A$ is viewed as a trivial $(X,+)$-module, it follows from Proposition \ref{trivial-cycle-group-cohomology} that the group $\Bilin((X,+) \times (X,+), A) \times \Hog^2_{sym}((X,+); A)$ classifies meta-trivial linear cycle sets, that is, extensions of a trivial linear cycle set by a trivial linear cycle set. 

\begin{question}
What can we say about the homomorphism $\Lambda$ for non-trivial linear cycle sets?
\end{question}
\bigskip

\section{Action of automorphisms on cohomology of linear cycle sets}\label{sec-action-auto-cohomo}

Let $(X, \cdot, +)$ be a linear cycle set and  $A$ an abelian group. Let $\Aut(X)$ denote the group of all linear cycle set automorphisms of $X$ and $\Autg(A)$ the usual automorphism group of $A$. For  $(\phi, \theta) \in \Aut(X) \times \Autg(A)$ and $(f, g) \in \Z_\Norm^2(X; A)$, we define 
$$^{(\phi, \theta)}(f, g)=\big(^{(\phi, \theta)}f, ~^{(\phi, \theta)}g \big),$$
where $^{(\phi, \theta)}f(x, y):=\theta \big( f \big(\phi^{-1}(x), \phi^{-1}(y) \big)\big)$ and $^{(\phi, \theta)}g(x, y):=\theta \big( g \big(\phi^{-1}(x), \phi^{-1}(y) \big)\big)$ for all $x, y \in X$.

\begin{proposition}
The group $\Aut(X) \times \Autg(A)$ acts by automorphisms on the group $\Ho_\Norm^2(X; A)$ as
$$^{(\phi, \theta)}[f, g]= [^{(\phi, \theta)}(f, g)]$$
for $(\phi, \theta) \in \Aut(X) \times \Autg(A)$ and $(f, g) \in \Z_\Norm^2(X; A)$.
\end{proposition}

\begin{proof}
For $(\phi, \theta) \in \Aut(X) \times \Autg(A)$ and $(f, g) \in \Z_\Norm^2(X; A)$, we first show that $^{(\phi, \theta)}(f, g)$ is a normalised 2-cocycle of the cycle set $X$. For $x, y \in X$, we have
\begin{eqnarray*}
^{(\phi, \theta)}f(x+y, z) &=& \theta \big( f \big(\phi^{-1}(x+y), \phi^{-1}(z) \big)\big)\\
&=& \theta \big( f \big(\phi^{-1}(x)+\phi^{-1}(y), \phi^{-1}(z) \big)\big)\\
&=& \theta \big( f \big(\phi^{-1}(x)\cdot \phi^{-1}(y), \phi^{-1}(x)\cdot \phi^{-1}(z) \big)\big) + \theta \big(f \big(\phi^{-1}(x), \phi^{-1}(z) \big)\big)~\textrm{due to}~\eqref{CocycleFull}\\
&=& \theta \big( f \big(\phi^{-1}(x \cdot y), \phi^{-1}(x \cdot z) \big)\big) + \theta \big(f \big(\phi^{-1}(x), \phi^{-1}(z) \big)\big) \\
&=& ^{(\phi, \theta)}f(x \cdot y, x \cdot z) +~ ^{(\phi, \theta)}f(x, z),
\end{eqnarray*}

\begin{eqnarray*}
&& ^{(\phi, \theta)}f(x, y+z)- ~^{(\phi, \theta)}f(x, y)- ~^{(\phi, \theta)}f(x, z)\\
&=& \theta \big( f \big(\phi^{-1}(x), \phi^{-1}(y)+\phi^{-1}(z)\big) - f \big( \phi^{-1}(x),  \phi^{-1}(y) \big)-f \big( \phi^{-1}(x),  \phi^{-1}(z) \big)\big)\\
&=& \theta \big( g \big(\phi^{-1}(x) \cdot \phi^{-1}(y), \phi^{-1}(x) \cdot \phi^{-1}(z)\big) - g \big( \phi^{-1}(y),  \phi^{-1}(z) \big)\big)~\textrm{due to}~\eqref{CocycleFull2}\\
&=& \theta \big( g \big(\phi^{-1}(x \cdot y), \phi^{-1}(x \cdot z)\big)\big) - \theta \big(g \big( \phi^{-1}(y),  \phi^{-1}(z) \big)\big)\\
&=& ^{(\phi, \theta)}g(x \cdot y, x \cdot z)-~ ^{(\phi, \theta)}g(y, z)
\end{eqnarray*}
and
\begin{eqnarray*}
^{(\phi, \theta)}g(x,y)+~^{(\phi, \theta)}g(x+y,z) &=& \theta \big( g\big(\phi^{-1}(x), \phi^{-1}(y) \big)+ g\big(\phi^{-1}(x)+\phi^{-1}(y), \phi^{-1}(z) \big) \big)\\
&=& \theta \big( g\big(\phi^{-1}(y), \phi^{-1}(z) \big)+ g\big(\phi^{-1}(x), \phi^{-1}(y)+\phi^{-1}(z) \big) \big)~\textrm{due to}~\eqref{CocycleFull3}\\
&=& \theta \big( g\big(\phi^{-1}(y), \phi^{-1}(z) \big)\big)+ \theta \big(g\big(\phi^{-1}(x), \phi^{-1}(y+z) \big) \big)\\
&=& ^{(\phi, \theta)}g(y, z) + ~^{(\phi, \theta)}g(x, y+ z).
\end{eqnarray*}
Further, $^{(\phi, \theta)}g(x,y)=~^{(\phi, \theta)}g(y,x)$ and $^{(\phi, \theta)}g(0, 0)=0$. Hence, $^{(\phi, \theta)}(f, g) \in \Z_\Norm^2(X; A)$. Now, given $(\phi_i, \theta_i) \in \Aut(X) \times \Autg(A)$ for $i=1,2$, we see that
\begin{eqnarray*}
^{(\phi_1, \theta_1)(\phi_2, \theta_2)}f(x, y) &=& ^{(\phi_1\phi_2, \theta_1\theta_2)}f(x, y)\\
&=& \theta_1\theta_2 \big( f \big(\phi_2^{-1}\phi_1^{-1}(x), \phi_2^{-1}\phi_1^{-1}(y)\big) \big)\\
&=& \theta_1 \big(~^{(\phi_2, \theta_2)}f \big(\phi_1^{-1}(x), \phi_1^{-1}(y)\big) \big)\\
&=& ^{(\phi_1, \theta_1)}\big({^{(\phi_2, \theta_2)}f}\big)(x, y) 
\end{eqnarray*}
for all $x, y \in X$. Similarly, one can show that $^{(\phi_1, \theta_1)(\phi_2, \theta_2)}g(x, y)=~^{(\phi_1, \theta_1)}\big({^{(\phi_2, \theta_2)}g}\big)(x, y)$. It can now easily be deduced that the group
$\Aut(X) \times \Autg(A)$ acts on $\Z_\Norm^2(X; A)$ by automorphisms. It only remains to be shown that the action preserves $\B_\Norm^2(X; A)$. Let $(f, g) \in \B_\Norm^2(X; A)$. Then there exists $\lambda: X \to A$ such that conditions \eqref{Coboundary1} and \eqref{Coboundary2} holds. For $x, y \in X$, we have
\begin{eqnarray*}
^{(\phi, \theta)}f(x,y) &=& \theta \big( f \big(\phi^{-1}(x), \phi^{-1}(y) \big)\big)\\
&=& \theta \big( \lambda \big(\phi^{-1}(x) \cdot \phi^{-1}(y) \big) -\lambda \big(\phi^{-1}(y)\big) \big)\\
&=& \theta \big( \lambda \big(\phi^{-1}(x \cdot y) \big) \big) - \theta \big( \lambda \big(\phi^{-1}(y)\big) \big)\\
&=& \lambda'(x \cdot y)- \lambda'(y)
\end{eqnarray*}
and
\begin{eqnarray*}
^{(\phi, \theta)}g(x,y) &=& \theta \big( \lambda \big(\phi^{-1}(x) + \phi^{-1}(y) \big) -\lambda \big(\phi^{-1}(x) \big)-\lambda \big(\phi^{-1}(y) \big) \big)\\
&=& \lambda'(x+y)-\lambda'(x)-\lambda'(y),
\end{eqnarray*}
where $\lambda'=\theta \lambda \phi^{-1}: X \to A$. Hence,  $^{(\phi, \theta)}(f, g) \in \B_\Norm^2(X; A)$ and we are done.
\end{proof}

Applying the orbit-stabiliser theorem to the action of $\Aut(X) \times \Autg(A)$ on $\Ho_\Norm^2(X; A)$ yields

\begin{corollary}\label{cor-order-cohom}
If $X$ is a finite linear cycle set, $A$ a finite abelian group and $(f, g)\in \Z_\Norm^2(X; A)$, then 
$$|\Ho_\Norm^2(X; A)| \ge \frac{|\Aut(X) \times \Autg(A)|}{|{(\Aut(X) \times \Autg(A))}_{[f, g]}|},$$
where ${(\Aut(X) \times \Autg(A))}_{[f, g]}$ is the stabiliser subgroup of $\Aut(X) \times \Autg(A)$ at $[f, g]$.
\end{corollary}

In general, we have  $\Aut(X) \le \Autg((X, +))$ for any linear cycle set  $(X, \cdot, +)$ and the equality holds for trivial linear cycle sets. Corollary \ref{cor-order-cohom} and Proposition \ref{trivial-cycle-group-cohomology} then yields
the following.

\begin{corollary}
If $G$ and $A$ are finite abelian groups with $A$ viewed as a trivial $G$-module and $(f, g)\in \Z_\Norm^2(G; A)$, then 
$$|\Hog_{sym}^2(G; A)| \ge \frac{|\Autg(G) \times \Autg(A)|}{|\Bilin(G \times G, A)|~|{(\Autg(G) \times \Autg(A))}_{[f, g]}|},$$
where ${(\Autg(G) \times \Autg(A))}_{[f, g]}$ is the stabiliser subgroup of $\Autg(G) \times \Autg(A)$ at $[f, g]$.
\end{corollary}

\medskip

\section{An exact sequence relating automorphisms and cohomology}\label{sec-exact-seq}
Let $(X, \cdot, +)$ be a linear cycle set,  $A$ an abelian group and 
\begin{equation}\label{central-extension}
\mathcal{E}: 0 \to A \overset{i}{\to} E \overset{\pi}{\to} X \to 0
\end{equation}
a central extension of $X$ by $A$. In view of Lemma \ref{extension-to-cocycle}, there exists a normalised 2-cocycle $(f, g)$ such that $$E \cong X \oplus_{f,g} A.$$ Further, the extension $\mathcal{E}$ determines a unique cohomology class $[f, g] \in \Ho_\Norm^2(X; A)$.
\par

Fix a central extension \eqref{central-extension} of a linear cycle set $(X, \cdot, +)$ by an abelian group $A$ and its corresponding cohomology class $[f, g] \in \Ho_\Norm^2(X; A)$ as in Lemma \ref{extension-to-cocycle}. For each $(\phi, \theta) \in \Aut(X) \times \Autg(A)$, we have $^{(\phi, \theta)}[f, g] \in \Ho_\Norm^2(X; A)$. Since the group $\Ho_\Norm^2(X; A)$ acts freely and transitively on itself by (left) translation, there exists a unique element $\Theta_{[f, g]} (\phi, \theta)  \in \Ho_\Norm^2(X; A)$ such that 
$$\Theta_{[f, g]} (\phi, \theta)  +~ ^{(\phi, \theta)}[f, g]= [f, g].$$
This gives a map
\begin{equation}\label{wells-map}
\Theta_{[f, g]}:  \Aut(X) \times \Autg(A) \to \Ho_\Norm^2(X; A).
\end{equation}
 We denote $\Theta_{[f, g]}$ by $\Theta$ for convenience of notation. Our aim is to relate certain group of automorphisms of  $E$ to groups $\Aut(X)$, $\Autg(A)$, $\Ho_\Norm^2(X; A)$ and $\Z_\Norm^1(X; A)$. For this purpose, we define
\begin{small}
$$\Aut_A(E) = \big\{\psi \in \Aut(E)~|~ \psi(x, a)=(\phi(x), \lambda(x) +\theta(a))~\textrm{for some}~(\phi, \theta) \in \Aut(X) \times \Aut(A)~\textrm{and map}~\lambda: X \to A \big\} .$$
\end{small} 

\begin{proposition}
$\Aut_A(E)$ is a subgroup of $\Aut(E)$.
\end{proposition}

\begin{proof}
Let $\psi_i(x, a)=(\phi_i(x), \lambda_i(x) +\theta_i(a))$ for $i=1, 2$. Then
\begin{eqnarray}
\nonumber \psi_1\psi_2 (x, a) &=&\psi_1 \big(\phi_2(x), ~\lambda_2(x) +\theta_2(a) \big)\\
\nonumber  &=&\big(\phi_1\phi_2(x), ~ \lambda_1\big(\phi_2(x)\big) + \theta_1\big(\lambda_2(x)\big) +\theta_1\theta_2(a) \big)\\
&=& \big(\phi_1\phi_2(x), ~ \lambda(x) +\theta_1\theta_2(a) \big)\label{psi1-psi2},
\end{eqnarray}
where $\lambda: X \to A$ is given by 
\begin{equation}\label{lambda1-lambda2}
\lambda(x)=\lambda_1\big(\phi_2(x)\big) + \theta_1\big(\lambda_2(x)\big)
\end{equation}
for $x \in X$ and $a \in A$. Consequently, $\Aut_A(E)$ is closed under composition. It remains to show that if $\psi \in \Aut_A(E)$, then $\psi^{-1} \in \Aut_A(E)$. But a direct computation shows that if $\psi(x, a)=(\phi(x), \lambda(x) +\theta(a))$, then $$\psi^{-1}(x, a)=\big(\phi^{-1}(x),~ \theta^{-1}\big(-\lambda(\phi^{-1}(x))\big) + \theta^{-1}(a) \big),$$ and hence $\Aut_A(E)$ is a subgroup of $\Aut(E)$.
\end{proof}

\begin{proposition}\label{condition for automorphism}
Let $(\phi, \theta) \in \Aut(X) \times \Autg(A)$ and $\lambda: X \to A$ a map. Then the map $\psi:E \to E $ given by $\psi(x, a)=(\phi(x), \lambda(x) +\theta(a))$ is an automorphism of $E$ if and only if 
\begin{equation}\label{homo-eq1}
\lambda(x + y) + \theta\big(g(x, y)\big) = \lambda(x) +\lambda(y) +  g\big(\phi(x), \phi(y)\big)
\end{equation}
and 
\begin{equation}\label{homo-eq2}
\lambda(x \cdot y) + \theta\big(f(x, y)\big)= \lambda(y)  + f\big(\phi(x), \phi(y)\big)
\end{equation}
for all $x, y \in X$. 
\end{proposition}

\begin{proof}
A direct computation shows that
\begin{eqnarray*}
\psi \big((x, a) +(y, b) \big) &=& \psi \big(x + y,~ a + b + g(x, y) \big)\\
&=& \big(\phi(x + y), ~\lambda(x + y) + \theta\big(a + b + g(x, y)\big) \big),\\
& &\\
\psi(x, a) +\psi (y, b) &=& \big(\phi(x), ~\lambda(x) +\theta(a)\big) + \big(\phi(y), \lambda(y) +\theta(b) \big)\\
&=& \big(\phi(x) + \phi(y),~ \lambda(x) +\lambda(y) + \theta(a) + \theta(b) + g\big(\phi(x), \phi(y)\big)\big),\\
& &\\
\psi \big((x, a) \cdot (y, b) \big) &=& \psi \big(x \cdot y, ~b + f(x, y) \big)\\
&=& \big(\phi(x \cdot y), ~\lambda(x \cdot y) + \theta\big(b + f(x, y)\big) \big),~\textrm{and} \\
& &\\
\psi(x, a) \cdot \psi (y, b) &=& \big(\phi(x), ~\lambda(x) +\theta(a)\big) \cdot \big(\phi(y), \lambda(y) +\theta(b) \big)\\
&=& \big(\phi(x) \cdot \phi(y), ~\lambda(y) + \theta(b) + f\big(\phi(x), \phi(y)\big)\big).
\end{eqnarray*}
The result now follows immediately from the preceding equalities.
\end{proof}

In view of \eqref{psi1-psi2}, the map
$$\Psi: \Aut_A(E) \to \Aut(X) \times \Autg(A)$$
given by 
$\Psi (\psi)=(\phi, \theta)$  is a group homomorphism.

\begin{proposition}\label{im-phi-ker-theta}
$\im(\Psi)=\Theta^{-1}\{0\}$.
\end{proposition}

\begin{proof}
First note that 
$$\Theta^{-1}\{0\}=\big\{(\phi, \theta) \in \Aut(X) \times \Autg(A)~|~   ^{(\phi, \theta)}[f, g]= [f, g]\big\},$$
the stabiliser subgroup of $\Aut(X) \times \Autg(A)$ at $[f, g]$. Suppose that $(\phi, \theta) \in \Theta^{-1}\{0\}$. Then, by definition of cohomologous 2-cocycles,  there exists a map $\lambda:X \to A$ such that 
\begin{eqnarray*}
^{(\phi, \theta)}g(x,y)- g(x,y) &=& \lambda (x + y) - \lambda(x) - \lambda(y)~\textrm{and}\\
^{(\phi, \theta)}f(x,y)-f(x,y) &=& \lambda (x \cdot y) - \lambda(y)
\end{eqnarray*}
for all $x, y \in X$. The preceding equations can be written as
\begin{eqnarray}
\label{imp-eq1} \theta \big(g(x,y)\big)- g\big(\phi(x),\phi(y)\big) &=& \lambda \big(\phi(x) + \phi(y)\big) - \lambda\big(\phi(x)\big) - \lambda\big(\phi(y)\big)~\textrm{and}\label{CoboundaryNew2}\\
\label{imp-eq2} \theta \big(f(x,y)\big)-f\big(\phi(x),\phi(y)\big) &=& \lambda \big(\phi(x) \cdot \phi(y)\big) - \lambda\big(\phi(y)\big)\label{CoboundaryNew1}
\end{eqnarray}
for all $x, y \in X$. We define $\psi: E \to E$ by setting
$$\psi(x, a)= \big(\phi(x), ~-\lambda\big(\phi(x)\big)+\theta(a)\big)$$
for $x \in X$ and $a \in A$. Notice that equations \eqref{imp-eq1} and  \eqref{imp-eq2} are precisely equations \eqref{homo-eq1} and \eqref{homo-eq2}, respectively. Hence, it follows from Proposition \ref{condition for automorphism} that  $\psi \in \Aut_A(E)$. Since $\Psi (\psi)=(\phi, \theta)$, we get $\Theta^{-1}\{0\}\subseteq \im(\Psi)$. Conversely, if $(\phi, \theta) \in \im(\Psi)$, then there exists a map  $\lambda:X \to A$ such that $\psi:E \to E$ given by $\psi(x, a):= (\phi(x), ~\lambda(x)+ \theta(a))$ lies in $\Aut_A(E)$. Again, by Proposition \ref{condition for automorphism}, the map $\psi$ being a morphism of linear cycle sets gives equations \eqref{CoboundaryNew2} and \eqref{CoboundaryNew1}. As seen above, this implies that $^{(\phi, \theta)}[f, g]= [f, g]$, which completes the proof.
\end{proof}

\begin{proposition}\label{ker-phi}
$ \Z_\Norm^1(X; A) \cong \Ker(\Psi)$.
\end{proposition}

\begin{proof}
Observe that $\psi \in \Ker(\Psi)$ if and only if there exists a map $\lambda:X \to A$ such that $\psi(x,a)=(x,~\lambda(x)+a)$ for all $x \in X$ and $a \in A$. Now, $\psi$ is a morphism of linear cycle sets if and only if  conditions \eqref{CoboundaryNew1} and  \eqref{CoboundaryNew2} hold with $\phi=\id_X$ and $\theta=\id_A$. These conditions take the form
 \begin{eqnarray}
\lambda (x \cdot y) &=& \lambda\big(y)~\textrm{and}\label{CoboundaryNewTrivial1}\\
\lambda (x+y) &=&  \lambda\big(x) + \lambda(y)\label{CoboundaryNewTrivial2}
\end{eqnarray}
for $x, y \in X$, and hence $\lambda \in \Z_\Norm^1(X; A)$. Conversely, given $\lambda \in \Z_\Norm^1(X; A)$, we see that $\psi: E \to E$ defined as $\psi(x,a)=(x,~\lambda(x)+a)$ is an element of  $\Ker(\Psi)$. In view of \eqref{lambda1-lambda2}, it follows that the map  $$\iota:  \Z_\Norm^1(X; A) \to \Ker(\Psi)$$ given by $\iota(\lambda)=\psi$ is an isomorphism of groups.
\end{proof}

Combining \eqref{wells-map}, Proposition \ref{im-phi-ker-theta} and Proposition \ref{ker-phi} gives the following exact sequence

\begin{theorem}\label{abelian-main-theorem}
Let $X$ be a linear cycle set, $A$ an abelian group and $E=X \oplus_{f, g} A$ the central extension of $X$ by $A$ corresponding to the 2-cocycle $(f, g)\in \Z_\Norm^2(X; A)$. Then there exists an exact sequence of groups
\begin{equation}\label{abelian-well-sequence}
1 \longrightarrow \Z^1(X;A) \stackrel{\iota}{\longrightarrow} \Aut_A(E) \stackrel{\Psi}{\longrightarrow} \Aut(X) \times \Autg(A) \stackrel{\Theta}{\longrightarrow} \Ho_\Norm^2(X; A),
\end{equation}
where exactness at $ \Aut(X) \times \Autg(A)$ means that $\im(\Psi)=\Theta^{-1}\{0\}$.
\end{theorem}

\begin{corollary}
Let $X$ be a linear cycle set and $A$ an abelian group such that $\Ho_\Norm^2(X; A)$ is trivial. Then every automorphism in $\Aut(X) \times \Autg(A)$ extends to an automorphism in $ \Aut_A(E)$.
\end{corollary}

Restricting the action of $\Aut(X) \times \Autg(A)$ on $\Ho_\Norm^2(X; A)$ to that of its subgroups $\Aut(X)$ and $\Autg(A)$ gives the following result.

\begin{corollary}
Every automorphism in $\Aut(X)_{[f]}$ and $\Autg(A)_{[g]}$ can be extended to an automorphism in $\Aut_A(E)$.
\end{corollary}
\bigskip

\section{Properties of map $\Theta$}\label{properties of map theta}
Let $(X, \cdot, +)$ be a linear cycle set and $A$ an abelian group. Since the group $\Aut(X) \times \Autg(A)$ acts on the group $\Ho^2_\Norm(X; A)$, we have their semi-direct product $\Ho^2_\Norm(X; A) \rtimes \big(\Aut(X) \times \Autg(A) \big)$. Further, the group $\Ho^2_\Norm(X; A)$ acts on itself by (left) translation.

\begin{proposition}
$\Ho^2_\Norm(X; A) \rtimes \big(\Aut(X) \times \Autg(A) \big)$ acts on $\Ho^2_\Norm(X; A)$ by setting
$$^{[\alpha](\phi, \theta)} [\beta]=~^{[\alpha]}{(^{(\phi, \theta)} [\beta])}$$
for $(\phi, \theta) \in  \Aut(X) \times \Autg(A)$ and $[\alpha], [\beta]\in  \Ho^2_\Norm(X; A)$.
\end{proposition}

\begin{proof}
For $ (\phi_1, \theta_1), (\phi_2, \theta_2) \in \Aut(X) \times \Autg(A)$ and $[\alpha_1], [\alpha_2], [\beta] \in \Ho^2_\Norm(X; A)$, we compute
\begin{eqnarray*}
^{\big([\alpha_1](\phi_1, \theta_1)\big)\big([\alpha_2](\phi_2, \theta_2)\big)}[\beta] & = & ^{\big([\alpha_1]~^{(\phi_1, \theta_1)}[\alpha_2] \big) \big((\phi_1, \theta_1)(\phi_2, \theta_2)\big)}[\beta]\\
& = & ^{\big([\alpha_1]~^{(\phi_1, \theta_1)}[\alpha_2]\big)}{\big(^{\big((\phi_1, \theta_1)(\phi_2, \theta_2)\big)}[\beta] \big)}\\
& = & ^{[\alpha_1]}{\big(^{^{(\phi_1, \theta_1)}{[\alpha_2]}}{\big(^{(\phi_1, \theta_1)}{\big(^{(\phi_2, \theta_2)}[\beta]\big)} \big)} \big)}\\
& = & ^{[\alpha_1]}{\big({^{(\phi_1, \theta_1)}{[\alpha_2]}} +{^{(\phi_1, \theta_1)}{\big(^{(\phi_2, \theta_2)}[\beta]} \big)} \big)},\\
& &  \textrm{since}~^{(\phi_1, \theta_1)}{[\alpha_2]} \in  \Ho^2_\Norm(X; A),~\textrm{which acts on itself by translation}\\
& = & ^{[\alpha_1]}{\big({^{(\phi_1, \theta_1)}{\big([\alpha_2] +~^{(\phi_2, \theta_2)}[\beta]}} \big)\big)},\\
&  & \textrm{since}~\Aut(X) \times \Autg(A)~\textrm{acts by automorphisms on}~\Ho^2_\Norm(X; A)\\
& = & ^{\big([\alpha_1](\phi_1, \theta_1)\big)}{\big(^{\big([\alpha_2](\phi_2, \theta_2)\big)}[\beta] \big)},\\
& & ~ \textrm{since}~\Ho^2_\Norm(X; A)~\textrm{acts on itself by translation}.
\end{eqnarray*}
Hence, $\Ho^2_\Norm(X; A) \rtimes \big(\Aut(X) \times \Autg(A) \big)$ acts on $\Ho^2_\Norm(X; A)$.
\end{proof}
\medskip

Let $G$ be a group and $A$ an abelian group equipped with an action of $G$. Then 
$$\Zg^1(G; A) =  \big\{f:G  \to A~|~ f(xy)= f(x)+~{^{x}}f(y)~ \textrm{for all}\ x,y\in G\big\}$$
is called the group of {\it 1-cocycles} and
$$\Bg^1(G;A) = \big\{f:G\to A~|~\textrm{there exists}~a \in A~\textrm{such that}~ f(x)={^x}a-a~\textrm{for all}~x\in G \big\}$$
the group of  {\it 1-coboundaries}  \cite[Chapter 4]{Brown1981}. Further, a complement of a subgroup $H$ in a group $G$ is another subgroup $K$ of $G$ such that $G=HK$ and $H\cap K=1$.  The following result relating 1-cocycles and complements is well-known \cite[11.1.2]{Robinson1982}.

\begin{lemma}\label{1-cocycle-complement}
Let $H$ be an abelian group and $G$ a group acting on $H$ by automorphisms. Then the map $f \mapsto \{f(g)g~|~g \in G \}$ gives a bijection from the set $\Zg^1(G; H)$ of 1-cocycles to the set $\{K \mid G=HK~\textrm{and}~H \cap K=1 \}$ of complements of $H$ in $G$. 
\end{lemma}

\begin{theorem}\label{main-thm-3}
Let $(X, \cdot, +)$ be a linear cycle set, $A$ an abelian group and $\Theta_{[\alpha]}: \Aut(X) \times \Autg(A) \longrightarrow \Ho^2_\Norm(X; A)$ the map corresponding to a cohomology class $[\alpha]\in \Ho^2_\Norm(X; A)$. Then the following hold:
\begin{enumerate}
\item $\Theta_{[\alpha]}$ is a group theoretical 1-cocycle. 
\item Any two such maps corresponding to distinct linear cycle set cohomology classes are cohomologous as group theoretical 1-cocycles.
\end{enumerate}
\end{theorem}

\begin{proof}
Suppose that $\Theta=\Theta_{[\alpha]}$ for $[\alpha] \in \Ho^2_\Norm(X; A)$ and $\bold{g} \in \Ho^2_\Norm(X; A) \rtimes \big( \Aut(X) \times \Autg(A)\big)$. Then for elements $[\alpha], ~^{\bold{g}}[\alpha] \in \Ho^2_\Norm(X; A)$, there exists a unique $[\beta] \in \Ho^2_\Norm(X; A)$ such that $~^{[\beta]}[\alpha]=~^{\bold{g}}[\alpha]$. Viewing $[\beta]$ as an element of $\Ho^2_\Norm(X; A) \rtimes \big( \Aut(X) \times \Autg(A)\big)$, it follows that $[\beta]^{-1}\bold{g} \in \mathbb{S}_{[\alpha]}$, the stabiliser subgroup of $\Ho^2_\Norm(X; A) \rtimes \big( \Aut(X) \times \Autg(A) \big)$ at $[\alpha]$, and hence
$$ \Ho^2_\Norm(X; A) \rtimes \big( \Aut(X) \times \Autg(A) \big)= \Ho^2_\Norm(X; A)\mathbb{S}_{[\alpha]}.$$ Further, since $\Ho^2_\Norm(X; A)$ acts freely on itself, it follows that $\mathbb{S}_{[\alpha]}$ is a complement of $\Ho^2_\Norm(X; A)$ in $\Ho^2_\Norm(X; A) \rtimes \big(\Aut(X) \times \Autg(A) \big)$. By Lemma \ref{1-cocycle-complement}, let $f:\Aut(X) \times \Autg(A) \to \Ho^2_\Norm(X; A)$ be the unique 1-cocycle corresponding to the complement $\mathbb{S}_{[\alpha]}$ of $\Ho^2_\Norm(X; A)$ in $\Ho^2_\Norm(X; A) \rtimes \big(\Aut(X) \times \Autg(A)\big)$. Then 
$$\mathbb{S}_{[\alpha]}= \big\{f(\phi, \theta)(\phi, \theta)~|~(\phi, \theta) \in \Aut(X) \times \Autg(A) \big\},$$
that is, $$[\alpha]=~^{f(\phi, \theta)(\phi, \theta)}[\alpha]=~^{f(\phi, \theta)}{\big(^{(\phi, \theta)}[\alpha]\big)}.$$
Now, by definition of $\Theta$ as in \eqref{wells-map}, we obtain $f(\phi, \theta)=\Theta (\phi, \theta)$, and hence $\Theta$ is a 1-cocycle.
\par

For the second assertion, let $\Theta=\Theta_{[\alpha]}$ and $\Theta'=\Theta'_{[\alpha']}$ for $[\alpha], [\alpha'] \in \Ho^2_\Norm(X; A)$. Then for any $(\phi, \theta) \in \Aut(X) \times \Autg(A)$, we have 
$$ ^{\Theta(\phi, \theta)} \big(^{(\phi, \theta)}[\alpha] \big)= [\alpha]~\textrm{and}~ ^{\Theta'(\phi, \theta)} \big(^{(\phi, \theta)}[\alpha'] \big)= [\alpha'].$$
Since $\Ho^2_\Norm(X; A)$ acts transitively on itself by (left) translation, there exists a unique $[\beta] \in \Ho^2_\Norm(X; A)$ such that $ ^{[\beta]}[\alpha']=[\alpha]$. This gives 
$$ ^{\Theta(\phi, \theta)} \big(^{^{(\phi, \theta)}{[\beta]}}{\big(^{(\phi, \theta)}[\alpha'] \big)} \big)= ~^{[\beta]}{\big(^{\Theta'(\phi, \theta)} \big(^{(\phi, \theta)}[\alpha'] \big)\big)}.$$
Since $[\beta], ~\Theta(\phi, \theta), {^{(\phi, \theta)}{[\beta]}}$ and $\Theta'(\phi, \theta)$ all lie in $\Ho^2_\Norm(X; A)$, which acts freely on itself, we must have $$\Theta(\phi, \theta) + {^{(\phi, \theta)}{[\beta]}}= \Theta'(\phi, \theta) + [\beta].$$
Thus, $\Theta$ and $\Theta'$ differ by a 1-coboundary, which completes the proof.
\end{proof}
\medskip

\section{Comparison with Wells exact sequence for groups}\label{comparision-sequences}
In \cite{Wells1971}, Wells derived an exact sequence relating 1-cocycles, automorphisms and second cohomology of groups corresponding to a given extension of groups. The sequence has found applications in some long standing problems on automorphisms of finite groups. We refer the reader to \cite[Chapter 2]{PSY2018} for a detailed account of the same, and recall the construction of this exact sequence for central extension of groups. Consider a central extension
\begin{equation}\label{extension-groups}
\mathcal{E}':  0 \to N \to G \to H \to 0
\end{equation}
of (additively written) groups. In this case $N$ is a trivial $H$-module, and hence the group  $\Zg^1(H;N)$ of 1-cocycles is simply the group of all homomorphisms from $H$ to $N$. Let $\Hog^2(H;N)$ be the second group cohomology of $H$ with coefficients in $N$, and $g:H \times H \to N$ be a group theoretical normalised 2-cocycle corresponding to the extension \eqref{extension-groups}. It follows from classical extension theory of groups that $$G \cong H \times_g N,$$ where $H \times_g N$ has underlying set $H \times N$ and group operation
$$(x, a)+(y, b)=(x+y, a +b+ g(x, y))$$
for $x, y \in H$ and $a, b \in N$. 
\par

Let $\Autg_N(G)$ be the group of automorphisms of $G$ keeping $N$ invariant as a set. In view of the identification $G \cong H \times_g N$, we have
\begin{Small}
$$\Autg_N(G) = \big\{\psi \in \Aut(G)~|~ \psi(x, a)=(\phi(x), \lambda(x) +\theta(a))~\textrm{for some}~(\phi, \theta) \in \Aut(H) \times \Aut(N)~\textrm{and map}~\lambda: H \to N \big\}.$$
\end{Small}
There is a monomorphism of groups $$\j: \Zg^1(H, N) \to  \Autg_N(G)$$ given by $\j(\lambda)=\psi$, where $\psi(x, a)=(x, \lambda(x)+a)$ \cite[Proposition 2.45]{PSY2018}. Also, there is a natural homomorphism $$\Phi: \Autg_N(G) \to  \Autg(H) \times \Autg(N)$$ given by $$\Phi(\psi)=(\phi, \theta).$$  As in Section \ref{sec-action-auto-cohomo}, there is an action of $\Autg(H) \times \Autg(N)$ on $\Hog^2(H;N)$. In fact, given any $(\phi, \theta) \in \Autg(H) \times \Autg(N)$ and $[h] \in \Hog^2(H;N)$, setting $$^{(\phi, \theta)} [h]=[^{(\phi, \theta)} h],$$ where  $^{(\phi, \theta)} h(x, y)= \theta \big( h \big(\phi^{-1}(x), \phi^{-1}(y) \big)\big)$ for $x, y \in X$, defines this action. Further, the action restricts to an action on the subgroup $\Hog^2_{sym}(H;N)$  of $\Hog^2(H;N)$ consisting of symmetric cohomology classes. Notice that the group $\Hog^2(H; N)$ acts freely and transitively on itself by (left) translation. Now, for each $(\phi, \theta) \in \Autg(H) \times \Autg(N)$, we have  cohomology classes $^{(\phi, \theta)}[g], [g] \in \Hog^2(H; N)$. Thus, there exists a unique element $\Omega (\phi, \theta)  \in \Hog^2(H; N)$ such that 
$$\Omega (\phi, \theta) + ~^{(\phi, \theta)}[g]= [g].$$
This gives a map
\begin{equation}\label{group-wells-map}
\Omega:  \Autg(H) \times \Autg(N) \to \Hog^2(H; N),
\end{equation}
which depends on the equivalence class of the extension $\mathcal{E}'$ or equivalently on its corresponding cohomology class. Further, $\Omega$ is a 1-cocycle with respect to the action of $ \Autg(H) \times \Autg(N)$ on  $\Hog^2(H; N)$ \cite[Corollary 2.41]{PSY2018}.  With the preceding set-up, Wells derived the following exact sequence of groups
\begin{equation}\label{central-group-wells}
1 \longrightarrow \Zg^1(H, N) \stackrel{\j}{\longrightarrow} \Autg_N(G) \stackrel{\Phi}{\longrightarrow} \Autg(H) \times \Autg(N) \stackrel{\Omega}{\longrightarrow} \Hog^2(H;N).
\end{equation}
\par

Let $\mathcal{L}$ and $\mathcal{A}$ denote the categories of linear cycle sets and abelian groups, respectively. Then there is a forgetful functor $$\mathcal{F}: \mathcal{L} \to \mathcal{A}$$ that maps a linear cycle set to its underlying abelian group. The preceding discussion shows that the functor $\mathcal{F}$ induces a map from the exact sequence \eqref{abelian-well-sequence} to the exact sequence \eqref{central-group-wells}.

\begin{theorem}\label{comparison-wells}
Let $(X, \cdot, +)$ be a linear cycle set, $A$ an abelian group viewed as a trivial $(X, +)$-module and $E=X \oplus_{f, g} A$ the central extension corresponding to the 2-cocycle  $(f, g)\in \Z_\Norm^2(X; A)$. Then the following diagram of groups commutes
$$
\xymatrix{
1 \ar[r]  &\Z^1(X;A) \ar@{^{(}->}[d]^{inclusion} \ar[r]^{\iota} &\Aut_A(E) \ar[r]^{\Psi \hspace{1cm}} \ar@{^{(}->}[d]^{inclusion} &  \Aut(X) \times \Autg(A) \ar[r]^{\hspace{5mm} \Theta} \ar@{^{(}->}[d]^{inclusion} & \Ho_\Norm^2(X; A) \ar[d]^{\Lambda}\\
1 \ar[r] & \Zg^1((X, +); A) \ar[r]^{\j} & \Autg_A((E, +)) \ar[r]^{\Phi \hspace{1cm}} & \Autg((X, +)) \times \Autg(A) \ar[r]^{\hspace{5mm}  \Omega } & \Hog_{sym}^2((X, +); A),}
$$
where $\Lambda$ is as in Proposition \ref{homo-cycle-group-cohomology}.
\end{theorem}
\bigskip

\section{Extensions of bi-groupoids and dynamical cocycles}\label{dynamical-2-cocycle}
During the last decade many new examples of bi-groupoids, namely,  bi-racks, bi-quandles, braces, skew braces, linear cycle sets, etc, have been introduced in connection to virtual knot theory and solutions of the Yang--Baxter equation. Let $X$ and $S$ be two non-empty sets, $\alpha, \alpha' : X \times X \to \Maps(S \times S, S)$ and $\beta, \beta' : S \times S \to \Maps(X \times X, X)$ maps. Then $X \times S$ with the binary operations 
\begin{equation}
(x, s) \cdot (y, t)= \big( \beta_{s, t}(x, y), ~\alpha_{x, y}(s, t) \big)
\end{equation}
and
\begin{equation}
(x, s) * (y, t)= \big( \beta'_{s, t}(x, y), ~\alpha'_{x, y}(s, t) \big)
\end{equation}
forms a bi-groupoid. Naturally, to obtain a bi-groupoid of special type it is essential to have functions $\alpha, \alpha', \beta, \beta'$ with nice properties. Using defining axioms of a cycle set, we can deduce a generalisation of \cite[Lemma 2.1]{Vendramin2016}, which itself is a linear cycle set analogue of a similar result for quandles \cite[Lemma 2.1]{AndruskiewitschGrana2003}.

\begin{proposition}
Let $X$  and $S$ be two sets, $\alpha : X \times X \to \Maps(S \times S, S)$ and $\beta : S \times S \to \Maps(X \times X, X)$ two maps. Then the set $X \times S$ with the binary operation 
\begin{equation}
(x, s) \cdot (y, t)= \big( \beta_{s, t}(x, y), ~\alpha_{x, y}(s, t) \big)
\end{equation}
forms a  cycle set if and only if the following conditions hold:
\begin{enumerate}
\item the map $(y, t) \mapsto  \big( \beta_{s, t}(x, y),~ \alpha_{x, y}(s, t) \big)$ is a bijection  for all $(x, s) \in X \times S$,
\item $\beta_{\alpha_{x, y}(s, t), \alpha_{x, z}(s, q)}\Big(\beta_{s, t}(x, y),~ \beta_{s, q}(x, z) \Big)= \beta_{\alpha_{y, x}(t, s), \alpha_{y, z}(t, q)}\Big(\beta_{t, s}(y, x),~ \beta_{t, q}(y, z) \Big)$ 
\item[] and
\item[] $\alpha_{\beta_{s, t}(x, y), \beta_{s, q}(x, z)}\Big(\alpha_{x, y}(s, t), ~\alpha_{x, z}(s, q) \Big)= \alpha_{\beta_{t, s}(y, x), \beta_{t, q}(y, z)}\Big(\alpha_{y, x}(t, s), ~\alpha_{y, z}(t, q) \Big)$ for all $x, y, z \in X$ and $s, t, q \in S$.
\end{enumerate}
\end{proposition}

If one of the sets is a cycle set, then we obtain

\begin{corollary}
Let $(X, \cdot)$ be a cycle set, $S$ a set and $\alpha : X \times X \to \Maps(S \times S, S)$  a  map. Then the set $X \times S$ with the binary operation 
\begin{equation}
(x, s) \cdot (y, t)= \big( x \cdot y, ~\alpha_{x, y}(s, t) \big)
\end{equation}
forms a cycle set if and only if the following conditions hold:
\begin{enumerate}
\item the map $(y, t) \mapsto  \big( x \cdot y,~ \alpha_{x, y}(s, t) \big)$ is a bijection  for each $(x, s) \in X \times S$,
\item $\alpha_{x \cdot y, x \cdot z}\Big(\alpha_{x, y}(s, t), ~\alpha_{x, z}(s, q) \Big)= \alpha_{y \cdot x, y \cdot z}\Big(\alpha_{y, x}(t, s), ~\alpha_{y, z}(t, q) \Big)$ for all $x, y, z \in X$ and $s, t, q \in S$.
\end{enumerate}
\end{corollary}

A map $\alpha$ satisfying condition (2) of the preceding corollary is referred as a {\it dynamical cocycle} of $X$ with values in $S$ and the cycle set structure on $X \times S$  is called a {\it dynamical extension} of $X$ by $\alpha$.  One can prove a similar result for linear cycle sets. Before stating the result, note that condition \ref{E:Cyclic} in the definition of a linear cycle set is redundant \cite[Section 3]{MR3881192}.

\begin{proposition}
Let $X$ and $S$ be two sets, $\alpha, \alpha' : X \times X \to \Maps(S \times S, S)$ and $\beta, \beta' : S \times S \to \Maps(X \times X, X)$ maps. Then the set $X \times S$ with the binary operations 
\begin{equation}
(x, s)\cdot (y, t)= \big( \beta_{s, t}(x, y), ~\alpha_{x, y}(s, t) \big)
\end{equation}
and
\begin{equation}
(x, s) + (y, t)= \big( \beta'_{s, t}(x, y), ~\alpha'_{x, y}(s, t) \big)
\end{equation}
forms a  linear cycle set  if and only if the following conditions hold:
\begin{enumerate}
\item the map $(y, t) \mapsto  \big( \beta_{s, t}(x, y),~ \alpha_{x, y}(s, t) \big)$ is a bijection  for each $(x, s) \in X \times S$,
\item $\beta_{s, \alpha'_{y, z}(t, q)}\Big(x,~ \beta'_{t, q}(y, z) \Big)= \beta'_{\alpha_{x, y}(s, t), \alpha_{x, z}(s, q)}\Big(\beta_{s, t}(x, y),~ \beta_{s, q}(x, z) \Big)$ 
\item[] and
\item[] $\alpha_{x, \beta'_{t, q}(y, z)}\Big(s, ~\alpha'_{y, z}(t, q) \Big)= \alpha'_{\beta_{s, t}(x, y), \beta_{s, q}(x, z)}\Big(\alpha_{x, y}(s, t), ~\alpha_{x, z}(s, q) \Big)$,

\item $\beta_{\alpha'_{x, y}(s, t), q}\Big(\beta'_{s, t}(x, y),~ z \Big)= \beta_{\alpha_{x, y}(s, t), \alpha_{x, z}(s, q)}\Big(\beta_{s, t}(x, y),~ \beta_{s, q}(x, z) \Big)$ 
\item[] and
\item[] $\alpha_{\beta'_{s, t}(x, y), z}\Big(\alpha'_{x, y}(s, t), ~q \Big)= \alpha_{\beta_{s, t}(x, y), \beta_{s, q}(x, z)}\Big(\alpha_{x, y}(s, t), ~\alpha_{x, z}(s, q) \Big)$
\end{enumerate}
 for all $x, y, z \in X$ and $s, t, q \in S$.
\end{proposition}

As before, if one of the sets is already a linear cycle set, then we have

\begin{corollary}
Let $(X, \cdot, +)$ be a linear cycle set, $S$ a set and $\alpha, \alpha' : X \times X \to \Maps(S \times S, S)$ two maps. Then the set $X \times S$ with the binary operations 
\begin{equation}
(x, s)\cdot (y, t)= \big( x \cdot y, ~\alpha_{x, y}(s, t) \big)
\end{equation}
and
\begin{equation}
(x, s) + (y, t)= \big( x + y, ~\alpha'_{x, y}(s, t) \big)
\end{equation}
forms a  linear cycle set  if and only if the following conditions hold:
\begin{enumerate}
\item the map $(y, t) \mapsto  \big( x \cdot y,~ \alpha_{x, y}(s, t) \big)$ is a bijection  for each $(x, s) \in X \times S$, 
\item  $\alpha_{x, y+z}\Big(s, ~\alpha'_{y, z}(t, q) \Big)= \alpha'_{x \cdot y, x \cdot z}\Big(\alpha_{x, y}(s, t), ~\alpha_{x, z}(s, q) \Big)$,
\item  $\alpha_{x + y, z}\Big(\alpha'_{x, y}(s, t), ~q \Big)= \alpha_{x \cdot y, x \cdot z}\Big(\alpha_{x, y}(s, t), ~\alpha_{x, z}(s, q) \Big)$
\end{enumerate}
 for all $x, y, z \in X$ and $s, t, q \in S$.
\end{corollary}

The pair of maps $\alpha, \alpha' : X \times X \to \Maps(S \times S, S)$  satisfying conditions (2)-(4) of the preceding corollary is called a {\it dynamical cocycle} of the linear cycle set $X$ with values in $S$ and the linear cycle set structure on $X \times S$ is called the {\it dynamical extension} of $X$ by $S$. Taking $$\alpha_{x, y}(s, t)= t+f(x, y)$$ and $$\alpha'_{x, y}(s, t)= s + t + g(x, y)$$ for some 2-cocycle $(f,g) \in \Z_\Norm^2(X; A)$,  we see that the dynamical extension generalises the extension obtained in Lemma \ref{linear-cycle-structure}.
\medskip

\noindent\textbf{Acknowledgments.}
The authors are grateful to the referee for the elaborate report which substantially improved the clarity and exposition of the paper.  Bardakov is supported by the Ministry of Science and Higher Education of Russia (agreement No. 075-02-2020-1479/1). Singh is supported by the SwarnaJayanti Fellowship grants DST/SJF/MSA-02/2018-19 and SB/SJF/2019-20/04 and the Indo-Russian grant DST/INT/RUS/RSF/P-19.


\medskip


\begin{thebibliography}{1}
\bibitem{AndruskiewitschGrana2003}  N.  Andruskiewitsch and M. Gra\~{n}a, \textit{From racks to pointed Hopf algebras}, Adv. Math. 178 (2003), no. 2, 177--243.

\bibitem{BCJO2017} D. Bachiller, F. Ced{\'o}, E. Jespers and J. Okni{\'n}ski, \textit{A family of irretractable square-free solutions of the Yang--Baxter  equation}, Forum Math. 29 (2017), no. 6, 1291--1306.

\bibitem{Brown1981} K. S.  Brown, \textit{Cohomology of Groups}, Graduate Texts in Mathematics, Vol. 87 (Springer, New York, 1982), x+306 pp.

\bibitem{CJO2010} F. Ced{\'o}, E. Jespers and J. Okni{\'n}ski, \textit{Retractability of set theoretic solutions of the Yang--Baxter  equation}, Adv. Math.  224 (6) (2010), 2472--2484.

\bibitem{CJO2014} F. Ced{\'o}, E. Jespers and J. Okni{\'n}ski, \textit{Braces and the Yang--Baxter equation}, Comm. Math. Phys.  327 (1) (2014), 101--116.

\bibitem{CedoSmokVendramin} F. Ced{\'o}, A.  Smoktunowicz and L. Vendramin, \textit{Skew left braces of nilpotent type}, Proc. Lond. Math. Soc. (3) 118 (2019), no. 6, 1367--1392.

\bibitem{Chouraqui} F. Chouraqui, \textit{Garside groups and Yang--Baxter equation}, Comm. Algebra 38 (12) (2010), 4441--4460.

\bibitem{Dehornoy2015} P. Dehornoy, \textit{Set-theoretic solutions of the Yang--Baxter equation,  RC-calculus, and Garside germs}, Adv. Math.  282 (2015), 93--127.

\bibitem{Drinfeld} V. G. Drinfeld, \textit{On some unsolved problems in quantum group theory}, Quantum groups (Leningrad, 1990), 1--8, Lecture Notes in Math., 1510, Springer, Berlin, 1992.


\bibitem{EtingofGelaki1998} P. Etingof and S. Gelaki, \textit{A method of construction of finite-dimensional triangular semisimple Hopf algebras}, Math. Res. Lett. 5 (1998), 551--561.

\bibitem{ESS1999} P. Etingof, T. Schedler and A. Soloviev, \textit{Set-theoretical solutions to the quantum Yang--Baxter equation}, Duke Math. J. 100(2) (1999), 169--209.

\bibitem{FJSK} R. Fenn, M. Jordan-Santana and L. H. Kauffman, \textit{Biquandles and virtual links}, Topology Appl. 145 (2004), 157--175.


\bibitem{GI2018} T. Gateva-Ivanova, \textit{Set-theoretic solutions of the Yang--Baxter equation, braces, and  symmetric groups}, Adv. Math. 338 (2018), 649--701.

\bibitem{IvanovaBergh} T. Gateva-Ivanova and M. Van den Bergh, \textit{Semigroups of $I$-type}, J. Algebra 206 (1998), 97--112.

\bibitem{GorshkovNasybullov} I. Gorshkov and T. Nasybullov, \textit{Finite skew-braces with solvable additive group}, J. Algebra  574 (2021), 172--183.


\bibitem{KM2005} L. H. Kauffman and V. Manturov, \textit{Virtual biquandles}, Fund. Math. 188 (2005), 103--146.

\bibitem{LebedJKTR2018} V. Lebed, \textit{Applications of self-distributivity to Yang--Baxter operators and their cohomology}, J. Knot Theory Ramifications  27 (11) (2018), 1843012, 20 pp.

\bibitem{LebedMortier} V.  Lebed and A.  Mortier, \textit{Abelian quandles and quandles with abelian structure group}, J. Pure Appl. Algebra 225 (2021), no. 1, 106474, 22 pp. 

\bibitem{LV2016} V. Lebed and L. Vendramin,  \textit{Cohomology and extensions of braces}, Pacific J. Math. 284 (2016), no. 1, 191--212. 

\bibitem{LV2017} V. Lebed and L. Vendramin, \textit{Homology of left non-degenerate set-theoretic solutions to the  Yang--Baxter equation}, Adv. Math. 304 (2017), 1219--1261.
  
 
\bibitem{YanZhu2000} J.-H. Lu, M. Yan and Y.-C. Zhu, \textit{On the set-theoretical Yang--Baxter equation}, Duke Math. J. 104 (2000), 1--18.

\bibitem{Nasybullov2019} T. Nasybullov, \textit{Connections between properties of the additive and the multiplicative groups of a two-sided skew brace}, J. Algebra 540 (2019), 156--167.

\bibitem{PSY2018} I. B. S. Passi, M. Singh and M. K. Yadav, \textit{Automorphisms of finite groups}, Springer Monographs in Mathematics. Springer, Singapore, 2018. xix+217 pp.


\bibitem{Rump2005} W. Rump, \textit{A decomposition theorem for square-free unitary solutions of the  quantum Yang--Baxter equation}, Adv. Math. 193 (2005), 40--55.

\bibitem{MR3881192} W. Rump, \textit{A covering theory for non-involutive set-theoretic solutions to the Yang-Baxter equation}, J. Algebra 520 (2019), 136--170.

\bibitem{RumpJA2007} W. Rump, \textit{Braces, radical rings, and the quantum Yang--Baxter equation}, J. Algebra 307 (2007), 153--170.


\bibitem{Smok2018} A. Smoktunowicz, \textit{A note on set-theoretic solutions of the Yang--Baxter equation}, J. Algebra 500 (2018), 3--18.

\bibitem{Smoktunowicz2018} A. Smoktunowicz, \textit{On Engel groups, nilpotent groups, rings, braces and the Yang--Baxter  equation}, Trans. Amer. Math. Soc. 370 (2018), no. 9, 6535--6564.

\bibitem{Soloviev2000} A. Soloviev, \textit{Non-unitary set-theoretical solutions to the quantum Yang--Baxter  equation}, Math. Res. Lett.  7 (2000), 577--596.

\bibitem{Vendramin2016} L. Vendramin, \textit{Extensions of set-theoretic solutions of the Yang--Baxter  equation and a conjecture of Gateva-Ivanova}, J. Pure Appl. Algebra 220 (2016), 2064--2076.


\bibitem{Robinson1982}  D. J. S. Robinson, \textit{A course in the theory of groups}, Graduate Texts in Mathematics, Vol. 80 (Springer, New York, 1982), xvii+481 pp.

\bibitem {Wells1971} C. Wells, \textit{Automorphisms of group extensions}, Trans. Amer. Math. Soc. 155 (1971), 189--194.

\end{thebibliography}
\end{document}